\documentclass[10pt]{amsart}
\usepackage[all]{xy}
\usepackage{amsmath}
\usepackage{amsfonts}
\usepackage{amssymb}
\usepackage{amscd}
\usepackage{amsthm}
\usepackage{latexsym}
\usepackage{amsbsy}
\usepackage{graphicx}
\usepackage{epstopdf}
\AppendGraphicsExtensions{.gif}
\newtheorem{lemma}{Lemma}[section]
\newtheorem{proposition}{Proposition}[section]
\newtheorem{theorem}{Theorem}[section]

\begin{document}

\title{On the Scalar Curvature for the Noncommutative Four Torus}

\author{Farzad Fathizadeh} 

\thanks{E-mail address: ffathiz@uwo.ca}

\date{}

\maketitle

\begin{abstract}

The scalar curvature for the noncommutative four torus $\mathbb{T}_\Theta^4$, 
where its flat geometry is conformally perturbed by a Weyl factor, is computed by making the use 
of a noncommutative residue that involves integration over the 3-sphere. This method is more 
convenient since it does not require the rearrangement lemma and it is advantageous as it explains 
the simplicity of the final functions of one and two variables, which describe the curvature with the help of 
a modular automorphism. In particular, 
it readily allows to write the function of two variables as the sum of a finite difference 
and a finite product of the one variable function. The curvature formula is simplified for dilatons of the 
form $sp$, where $s $ is a real parameter and $p \in C^\infty(\mathbb{T}_\Theta^4)$ is an arbitrary projection, 
and it is observed that, in contrast to the two dimensional case studied by A. Connes and H. Moscovici, unbounded 
functions of the parameter $s$ appear in the final formula. An explicit formula for the gradient of the 
analog of the Einstein-Hilbert action is also calculated.

\end{abstract}

\tableofcontents

\section{Introduction}

The computation of scalar curvature \cite{conmos2, fatkha4} for noncommutative two tori $\mathbb{T}_\theta^2$ 
was stimulated by  
the seminal work \cite{contre} of A. Connes and P. Tretkoff on the Gauss-Bonnet theorem for these 
$C^*$-algebras, and its extension in \cite{fatkha1} to general translation-invariant conformal structures.  
Flat geometries of $\mathbb{T}_\theta^2$ \cite{con1, con1.5}, whose conformal classes are represented by  positive 
Hochschild cocycles \cite{con3}, are conformally perturbed by means of a positive invertible element  
$e^{-h}$, where $h=h^* \in C^\infty(\mathbb{T}_\theta^2)$ is a dilaton \cite{contre}. Local geometric invariants, 
such as scalar curvature, can then be computed by considering small time asymptotic expansions, 
which depend on the action of the algebra on a Hilbert space and the distribution at infinity of the 
eigenvalues of a relevant geometric operator, namely the Laplacian of the conformally perturbed metric.

Following these works, the local differential geometry of noncommutative tori equipped with curved metrics 
has received considerable attention in recent years \cite{bhumar, fatkha3, dabsit1, fatkha2, dabsit2}. See also 
\cite{ros, fatghokha}. It should be mentioned that conformal geometry in the noncommutative setting is intimately related to 
twisted spectral triples and we refer to \cite{conmos1.5, contre, conmos2, ponwan, mos1, gremarteh} for detailed discussions. 
Also it is closely related to the spectral action computations in the presence of a dilaton \cite{chacon1, conmar}. 
For noncommutative four tori $\mathbb{T}_\Theta^4$, 
the scalar curvature is computed in \cite{fatkha2} and it is shown that flat metrics are the critical points of the analog 
of the Einstein-Hilbert action. Also noncommutative residues for noncommutative tori were studied in \cite{fatwon, levjimpay, fatkha2} (see also \cite{sit1}). We refer to \cite{wod, guil} and \cite{lesjim} for detailed discussions on noncommutative residues for classical manifolds.

A crucial tool for the local computations on noncommutative tori has been Connes' pseudodifferential 
calculus, developed for $C^*$-dynamical systems in \cite{con1}, which can be employed to work in the 
heat kernel scheme of elliptic differential operators and index theory 
(cf. \cite{gil}). An obstruction in these 
calculations, which is a purely noncommutative feature, is the appearance of integrals of functions over the 
positive real line that are $C^*$-algebra valued. This is overcome by the rearrangement lemma 
\cite{contre, conmos2} (cf. \cite{bhumar, fatkha4, les}), which uses the modular automorphism of the state 
implementing the conformal perturbation   
and delicate Fourier analysis to reorder the integrands and computes the integrals explicitly. The integrals 
are then expressed as somewhat complicated functions of the modular  automorphism acting on 
relevant elements of the $C^*$-algebra. This lemma has been generalized in \cite{les}, and 
the work in \cite{dabsit2} is an instance where the generalization is used.

A striking fact about the final formulas for the curvature of noncommutative tori is their simplicity and their 
fruitful properties such as being entire. Considering the numerous functions from the rearrangement 
lemma that get involved in hundreds of terms in the computations, the final simplicity indicates an enormous 
amount of cancellations, which are carried out by computer assistance. One of the aims of this paper is 
to explain this simplicity by computing the scalar curvature for $\mathbb{T}_\Theta^4$ without using the 
rearrangement lemma. We then study the curvature formula for the dilatons that are associated with 
projections in $C^\infty(\mathbb{T}_\Theta^4)$. The gradient of the Einstein-Hilbert action is also 
calculated, which prepares the ground for studying its associated flow in future works.

This article is organized as follows. In \S \ref{preliminaries} we recall the formalism and notions used in 
\cite{fatkha2}  concerning the 
conformally perturbed Laplacian on $\mathbb{T}_\Theta^4$.  
In \S \ref{SCsection} we use a noncommutative residue that involves integrations on the 3-sphere 
$\mathbb{S}^3$ to compute the scalar curvature. This method is quite convenient as it does not require any help 
from the rearrangement lemma and it is advantageous as it explains the simplicity of the final formula \eqref{SCformula}.  
Also it readily allows to write the function of two variables in \eqref{SCformula} as the sum of a finite difference and a finite product 
of the one variable function.  In \S \ref{PRsection} the curvature formula is simplified for dilatons of the 
form $sp$, where $s $ is a real parameter and $p \in C^\infty(\mathbb{T}_\Theta^4)$ is an arbitrary projection.  
It is observed that, in contrast to the two dimensional case studied in \cite{conmos2}, unbounded 
functions of the parameter $s$ appear in the final formula. In \S \ref{GRsection}, we compute an explicit 
formula for the gradient 
of the analog of the Einstein-Hilbert action in terms of finite differences (cf. \cite{conmos2, les}) of a one variable function 
that describes this action \cite{fatkha2}.

\section{Preliminaries} \label{preliminaries}

The noncommutative four torus $C(\mathbb{T}_\Theta^4)$ is the universal $C^*$-algebra generated by 
four unitaries $U_1, U_2, U_3, U_4$, which satisfy the commutation relations 
\[
U_k U_j = e^{2 \pi i \theta_{k j}} U_j U_k, 
\]
where $\Theta=(\theta_{kj}) \in M_4(\mathbb{R})$ is an antisymmetric matrix. For simplicity 
elements of the form $U_1^{\ell_1}U_2^{\ell_2}U_3^{\ell_3}U_4^{\ell_4} \in C(\mathbb{T}_\Theta^4)$ shall be 
denote by $U^\ell$ for any 4-tuple of integers $\ell=(\ell_1, \ell_2, \ell_3, \ell_4)$. 

There is a natural action of 
$\mathbb{R}^4$ on this $C^*$-algebra, which is defined by 
\[
\alpha_s (U^\ell) = e^{i s\cdot \ell} U^\ell, \qquad s \in \mathbb{R}^4, \qquad \ell \in \mathbb{Z}^4, 
\]
and is extended to a 4-parameter family of $C^*$-algebra automorphisms 
$\alpha: \mathbb{R}^4 \to \textnormal{Aut}\big(C(\mathbb{T}_\Theta^4)\big)$. 
The infinitesimal generators of this action, denoted by $\delta_1, \delta_2, \delta_3, \delta_4$,  are defined 
on the smooth subalgebra 
\[
C^\infty(\mathbb{T}_\Theta^4) = \{ a \in C(\mathbb{T}_\Theta^4); \,\, \textnormal{the map} \,\, \mathbb{R}^4 \ni s \mapsto \alpha_s(a) \,\, \textnormal{is smooth}   \}, 
\]
which is a dense subalgebra of $C(\mathbb{T}_\Theta^4)$ and can alternatively be defined as the space of elements 
of the form $\sum_{\ell \in \mathbb{Z}^4} a_\ell U^\ell$ with rapidly decaying complex coefficients $(a_\ell) \in \mathcal{S}(\mathbb{Z}^4).$ These derivations are determined by the relations $\delta_j (U_j) = U_j$ and $\delta_j(U_k) = 0$, if $j \neq k$.

One can consider a complex structure on $C(\mathbb{T}_\Theta^4)$ (cf. \cite{fatkha2}) by introducing the analog of the Dolbeault operators  
\[ 
\partial_1 =  \delta_1 - i \delta_3, \qquad 
\partial_2=  \delta_2 - i \delta_4, \qquad 
 \bar \partial_1 =  \delta_1 + i \delta_3, 
\qquad \bar \partial_2= \delta_2 + i \delta_4, 
\] 
and by setting  
\[
\partial = \partial_1 \oplus \partial_2, \qquad \bar \partial = {\bar \partial}_1 \oplus  {\bar \partial}_2,
\] which are 
maps from $C^\infty(\mathbb{T}_\Theta^4)$ to $ C^\infty(\mathbb{T}_\Theta^4) \oplus C^\infty(\mathbb{T}_\Theta^4).$

There is a canonical positive faithful trace  $ \varphi_0:C(\mathbb{T}_\Theta^4)\to \mathbb{C}$, which is defined on the smooth algebra by 
\[
\varphi_0 \big (\sum_{\ell \in \mathbb{Z}^4} a_\ell U^\ell \big ) = a_0. 
\]
Following the method introduced in \cite{contre},  $\varphi_0$ is viewed as the volume form, and conformal 
perturbation of the metric is implemented in \cite{fatkha2} by choosing a dilaton $h=h^* \in C^\infty(\mathbb{T}_\Theta^4)$ 
and by considering the linear functional $ \varphi:C(\mathbb{T}_\Theta^4)\to \mathbb{C}$ given by 
\[
\varphi(a) = \varphi_0 (a e^{-2h}), \qquad  a \in C(\mathbb{T}_\Theta^4). 
\]
This is a KMS state and we consider the associated 1-parameter group 
$\{\sigma_t \}_{t \in \mathbb{R}}$ of inner automorphisms given by 
\[
\sigma_t (a) = e^{ith} a e^{-ith}, \qquad a \in C(\mathbb{T}_\Theta^4), 
\]
and will use the following operators substantially
\[
\Delta (a) = \sigma_i(a)=e^{-h} a e^h,  \qquad \nabla(a)= \log \Delta (a) = [-h, a], \qquad a \in C(\mathbb{T}_\Theta^4). 
\]
Denoting the inner product  associated with the state $\varphi$  by 
\begin{equation} 
(a, b)_\varphi = \varphi(b^*a), \qquad a,b \in C(\mathbb{T}_\Theta^4), \nonumber
\end{equation}
the Hilbert space completion of $C(\mathbb{T}_\Theta^4)$ 
with respect to this inner product is denoted by $\mathcal{H}_\varphi$, and 
the analog  of the de Rham differential  is defined in \cite{fatkha2} by  
\[
d = \partial \oplus \bar \partial : \mathcal{H}_\varphi \to 
\mathcal{H}_\varphi^{(1,0)} \oplus \mathcal{H}_\varphi^{(0,1)}. 
\]
The Hilbert spaces $\mathcal{H}_\varphi^{(1,0)}$ and $\mathcal{H}_\varphi^{(0,1)} $ 
are respectively the completions of the analogs of $(1, 0)$-forms and $(0, 1)$-forms, namely  the spaces  
$\{ \sum_{i=1}^n a_i \partial b_i; a_i, b_i \in C^\infty(\mathbb{T}_\Theta^4), n \in \mathbb{N} \}$ and 
$\{ \sum_{i=1}^n a_i \bar \partial b_i; a_i, b_i \in C^\infty(\mathbb{T}_\Theta^4), n \in \mathbb{N} \}$, with 
the appropriate inner product related to the conformal factor. 

The Laplacian $d^*d :  \mathcal{H}_\varphi  \to  \mathcal{H}_\varphi $ is then computed and shown to be 
anti-unitarily equivalent to the operator 
\[ \triangle_\varphi = 
e^{h} \bar \partial_1 e^{-h} \partial_1 e^{h} + 
e^{h} \partial_1 e^{-h} \bar \partial_1 e^{h} + 
e^{h} \bar \partial_2 e^{-h} \partial_2 e^{h} + 
e^{h}  \partial_2 e^{-h} \bar \partial_2 e^{h}. 
\]

\section{Scalar Curvature for $\mathbb{T}_\Theta^4$ and its Functional Relations} \label{SCsection}

The scalar curvature of the conformally perturbed metric on $\mathbb{T}_\Theta^4$ is the unique element 
$R \in C^\infty(\mathbb{T}_\Theta^4)$ such that 
\begin{equation} 
\textnormal{res}_{s=1} \textnormal{Trace}(a \triangle_{\varphi}^{-s}) = 
\varphi_0(a R),  \qquad \forall a \in C^\infty(\mathbb{T}_\Theta^4). \nonumber
\end{equation}
Since the linear functional 
\[
\int\!\!\!\!\!\!- \, P= \textnormal{res}_{s=0} \textnormal{Trace}(P \triangle_\varphi^{-s})
\]
defines a trace on the algebra of pseudodifferential operators \cite{conmos1, hig1}, it follows from 
the uniqueness of traces on the algebra of pseudodifferential operators \cite{fatwon, fatkha2, levjimpay} 
that it coincides with the noncommutative residue defined in \cite{fatkha2}. Therefore there exists a constant 
$c$ such that for any $P$
\[
\int\!\!\!\!\!\!- \, P = c\, \int_{\mathbb{S}^3} \varphi_0 (\rho_{-4}(\xi)) \, d\Omega, 
\]
where $\rho_{-4}$ is the homogeneous term of order $-4$ in the expansion of the symbol of $P$, and 
$d\Omega$ is the invariant measure on the sphere $\mathbb{S}^3$. Therefore, in order to compute 
the curvature, we can write 
\begin{eqnarray} 
\textnormal{res}_{s=1} \textnormal{Trace}(a \triangle_{\varphi}^{-s}) =  
\textnormal{res}_{s=0} \textnormal{Trace}(a \triangle_{\varphi}^{-s-1})  
= \int\!\!\!\!\!\!- \, a \triangle_{\varphi}^{-1} 
= c \, \varphi_0 \Big (\int_{\mathbb{S}^3} a \,b_2(\xi) \, d \Omega \Big ), \nonumber
\end{eqnarray}
where $b_j$ is the homogeneous term of order $-2-j$ in the asymptotic expansion of the symbol 
of the parametrix 
of $\triangle_\varphi$. Hence 
\[
 R = c \int_{\mathbb{S}^3}  \,b_2(\xi) \, d \Omega.  
\]
We compute $b_2$ by applying Connes' pseudodifferential calculus 
\cite{con1} to the symbol of $\triangle_\varphi$, which is the sum of the homogeneous components     
\begin{eqnarray}
a_2 (\xi) =  e^h \sum_{i=1}^4  \xi_i^2, \quad 
a_1 (\xi)=  \sum_{i=1}^4  \delta_i(e^h) \xi_i,  \quad 
 a_0(\xi) =  \sum_{i=1}^4 \big (  \delta_i^2(e^h) - \delta_i(e^h)e^{-h} 
\delta_i(e^h) \big ). \nonumber
\end{eqnarray}
That is, we solve the following equation explicitly up to $b_2$  
\[
(b_0 + b_1 + b_2 + \cdots) \circ (a_0 + a_1+ a_2) \sim 1,  
\]
which in general yields 
\[
b_0 = a_2^{-1} = \big ( e^h \sum_{i=1}^4  \xi_i^2 \big )^{-1},  \qquad
b_n  = - \sum_{\substack{2+j+|\ell |-k=n, \\ 0 \leq j <n, \, 0 \leq k \leq 2}} \frac{1}
{\ell !}
\partial^{\ell} (b_j)
\delta^{\ell} (a_k) b_0 \qquad (n >1). 
\]
Here, for any $\ell=(\ell_1, \ell_2, \ell_3, \ell_4) \in \mathbb{Z}_{\geq0}^4$, 
$\partial^\ell$  denotes $\partial_{\xi_1}^{\ell_1} \partial_{\xi_2}^{\ell_2} \partial_{\xi_3}^{\ell_3} \partial_{\xi_4}^{\ell_4} $ and 
$\delta^\ell$ denotes $\delta_1^{\ell_1}\delta_2^{\ell_2}\delta_3^{\ell_3}\delta_4^{\ell_4}$.
Note that the composition rule for pseudodifferential symbols \cite{con1}, 
\[ \rho \circ \rho' = 
\sum_{\ell \in \mathbb{Z}^4_{\geq 0} } 
\frac{1}{\ell ! }
\partial^{\ell} \rho (\xi) \,
\delta^{\ell} (\rho'(\xi)), 
\]
is used in the derivation of the above recursive formula for $b_n$. 

Computing $b_2$ and restricting it to 
$\mathbb{S}^3$ by the substitutions  
\begin{eqnarray}
&& \xi_1 = \cos(\psi)    ,  \quad \qquad  \qquad \qquad \,\,\,\,
\xi_2 =   \cos(\theta) \sin (\psi),  \nonumber \\ 
&&\xi_3 = \sin (\theta ) \cos (\phi ) \sin (\psi ),      \qquad   \,
\xi_4 = \sin (\theta ) \sin (\phi ) \sin (\psi ) ,      \nonumber
\end{eqnarray}
with $0 \leq \psi < \pi,$ $0 \leq \theta < \pi,$ 
$0 \leq \phi < 2 \pi,$ we perform its integral over the sphere and find that 
\begin{eqnarray}
\int_{\mathbb{S}^3} b_2 (\xi) \, d \Omega &=& 
\int_0^{2 \pi} \int_0^\pi \int_0^\pi  b_2 (\xi) \sin (\theta ) \sin ^2(\psi )\, d \psi \,  d\theta \, d \phi  \nonumber 
\end{eqnarray}
\begin{eqnarray}  \label{SC4formula1}
&=&\sum_{i=1}^4 \Big (\left(
-2 \pi ^2\right)b_0\delta _i\delta _i(e^{h})b_0+
\left(2 \pi^2\right)b_0\delta _i(e^{h})\frac{1}{e^{h}}\delta _i(e^{h})b_0+
\frac{5 \pi^2}{2}b_0\delta _i(e^{h})b_0\delta _i(e^{h})b_0 \nonumber \\
&&+\left(3 \pi^2\right)b_0e^{h}b_0\delta _i\delta _i(e^{h})b_0+
\left(-8 \pi^2\right)b_0e^{h}b_0\delta _i(e^{h})b_0\delta _i(e^{h})b_0 \nonumber \\
&&+
\left(-2 \pi^2\right)b_0e^{h}b_0e^{h}b_0\delta _i\delta _i(e^{h})b_0 
   +\left(-\pi^2\right)b_0\delta _i(e^{h})b_0e^{h}b_0\delta _i(e^{h})b_0 \nonumber \\
   &&+
   \left(2 \pi^2\right)b_0e^{h}b_0\delta _i(e^{h})b_0e^{h}b_0\delta _i(e^{h})b_0
   + 
   \left(4 \pi^2\right)b_0e^{h}b_0e^{h}b_0\delta _i(e^{h})b_0\delta _i(e^{h})b_0 \Big ) \nonumber \\
&&= \pi^2 \sum_{i=1}^{4} \Big ( -e^{-h} \delta_i^2(e^h) e^{-h} +\frac{3}{2} e^{-h} \delta_i(e^h) e^{-h}  \delta_i(e^h)e^{-h} \Big ). 
\end{eqnarray}
The fact that, over $\mathbb{S}^3$,  $b_0$ reduces to $e^{-h}$ is crucial in the last equation, which leads to 
such a simple final formula.

We then use the following identities  \cite{contre, conmos2, fatkha4} to write the expression  \eqref{SC4formula1} in terms of 
$\nabla=\log \Delta = - \textnormal{ad}_{h}$ and $\delta_i(h)$: 
\begin{eqnarray}
e^{-h}\delta_i(e^h) = g_1(\Delta) (\delta_i(h)),  \qquad
e^{-h} \delta_i^2(e^h) =  g_1(\Delta) (\delta_i^2(h)) +  
2g_2(\Delta, \Delta)(\delta_i(h) \delta_i(h)), \nonumber
\end{eqnarray}
where
\begin{equation} \label{gfunction}
g_1(u)=\frac{u-1}{\log u}, \qquad
g_2(u, v)= \frac{u (v-1) \log (u)-(u-1) \log (v)}{\log (u) \log (v) 
(\log (u)+\log(v))}. 
\end{equation}
This yields 
\begin{eqnarray} \label{SCformula}
\int_{\mathbb{S}^3} b_2 (\xi) \, d \Omega &=& \frac{1}{c} R \nonumber \\ &=& \pi^2 \sum_{i=1}^4 \Big ( - e^{-h} \Delta^{-1} g_1(\Delta) (\delta_i(h)) 
-2 e^{-h} \Delta^{-1} \big ( g_2(\Delta, \Delta)(\delta_i(h)^2) \big ) \nonumber \\
&& \qquad+ \frac{3}{2} e^{-h} \Delta^{-1} \big (  g_1(\Delta)(\delta_i(h)) g_1(\Delta)(\delta_i(h))  \big ) \Big ) \nonumber \\ 
&=& \pi^2 e^{-h} k(\nabla) \Big ( \sum_{i=1}^4 \delta_i^2(h) \Big) + \pi^2 e^{-h} H(\nabla, \nabla) \Big ( \sum_{i=1}^4 \delta_i(h)^2 \Big ), 
\end{eqnarray}
where 
\[
k(s)= - e^{-s} g_1(e^s) = \frac{e^{-s}-1}{s},
\]
\begin{eqnarray} \label{twovar}
H(s, t)&=& -2 e^{-s-t} g_2(e^s, e^t) + \frac{3}{2} e^{-s-t} g_1(e^s) g_1(e^t)  \nonumber \\
&=& \frac{e^{-s-t} \left(\left(e^s-1\right)
   \left(3 e^t+1\right) t-\left(e^s+3\right)
   s \left(e^t-1\right)\right)}{2 s t (s+t)}. 
\end{eqnarray}
This formula matches with the one obtained in \cite{fatkha2} (up to the multiplicative factor $1/c=2 \pi^2$). 

\begin{theorem}
Let 
\[
\tilde k (s) = e^s k(s), \qquad \tilde H (s, t) = e^{s+t} H(s, t), 
\] 
where $k$ and $H$ are the functions in the final formula for the scalar curvature. 
We have 
\begin{equation} \label{related}
\tilde H(s, t) = 2 \frac{\tilde k (s+t)- \tilde k(s)}{t} + \frac{3}{2} \tilde k(s) \tilde k (t).
\end{equation}
\end{theorem}

\begin{proof}
It follows from \eqref{twovar} and the following relation between the functions introduced in \eqref{gfunction}:
\begin{eqnarray}
g_2(u, v) &=& \int_0^1 s u^s  g_1(v^s) \, ds 
= \frac{1}{\log (v)} \Big ( \frac{uv-1}{\log ( uv)} - \frac{u-1}{\log (u)} \Big ) \nonumber \\
&=&\frac{1}{\log (v)} \big ( g_1(uv) - g_1(u) \big ). \nonumber
\end{eqnarray}
\end{proof}

\section{Projections and the Scalar Curvature} \label{PRsection}

Similar to the illustration in \cite{conmos2} of the scalar curvature of $\mathbb{T}_\theta^2$ for 
dilatons associated with projections, we consider dilatons of the form $h= s p$, where $s \in \mathbb{R}$ and 
$p=p^*=p^2 \in C^\infty(\mathbb{T}_\Theta^4)$ is an arbitrary projection, and simplify the expression \eqref{SCformula} 
for these cases. We shall also study the behaviour of the functions of the parameter $s$ that appear in the final formula.

\begin{proposition} \label{SCprojection}
Let $p=p^*=p^2 \in C^\infty(\mathbb{T}_\Theta^4)$ be a projection. For the dilaton $h= s p$, $s \in \mathbb{R}$, 
the formula for the scalar curvature reduces to 

\begin{eqnarray}
R=e^{-s p} \big ( f_1(s)  \triangle(p) + f_2(s)  \triangle(p) p  + f_3(s) p \triangle(p)
+  f_4(s)   p \triangle(p) p \big ), \nonumber
\end{eqnarray}
where $\triangle = \sum_{i=1}^4 \delta_i^2$ and
\begin{eqnarray}
&&f_1(s)=\frac{1}{4} (-2 \sinh (s)+\cosh (s)-1), \qquad f_2(s)= \frac{1}{2}\sinh ^2\left(\frac{s}{2}\right), \nonumber \\
&& f_3(s)= \frac{-s+\sinh (s)}{2}-\frac{\cosh (s)-1}{4}, \qquad f_4(s)= s-\sinh (s). \nonumber
\end{eqnarray}
\end{proposition}

\begin{proof}
Our method is quite similar to the one used in \cite{conmos2}. That is, we first use the identity 
\[
\triangle(p) = p \triangle(p) p + p \triangle(p) (1-p) + (1-p) \triangle(p) p + (1-p) \triangle(p) (1-p), 
\]
to decompose $\triangle (p) $ to the sum of eigenvectors of $\nabla =- \textnormal{ad}_{sp}$ with eigenvalues 
$0, -s, s, 0$. Therefore 
\begin{eqnarray}
k(\nabla)( \triangle(h) ) &=& s k(\nabla) (\triangle(p)) \nonumber \\
&=& sk(0) \big ( p \triangle(p) p + (1-p) \triangle(p) (1-p)  \big ) \nonumber \\
&&+ s k(-s) \big (  p \triangle(p) (1-p) \big )
+s k(s) \big (  (1-p) \triangle(p) p \big ) \nonumber \\
&=& - s \big ( p \triangle(p) p + (1-p) \triangle(p) (1-p)  \big ) \nonumber \\
&&+( 1-e^s)   \big (  p \triangle(p) (1-p) \big ) + (e^{-s}-1) \big (  (1-p) \triangle(p) p \big ). \nonumber
\end{eqnarray}
Then, using the identity $\delta_i (p) = \delta_i(p) p + p \delta_i(p)$, one can see that 
\begin{eqnarray}
&&H(\nabla, \nabla)(\delta_i(h) \delta_i(h)) \nonumber \\ 
&&= \frac{s^2}{2} \Big (\big ( H(s, -s)+H(-s, s) \big ) + \big ( H(s, -s)- H(-s, s) \big ) (1-2 p) \Big ) \big (\delta_i(p) \delta_i(p) \big ) \nonumber \\
&&=\frac{s^2}{2} \Big ( \frac{2 (\cosh (s)-1)}{s^2} + \frac{4 (s-\sinh (s))}{s^2} (1-2 p)   \Big ) \big (\delta_i(p) \delta_i(p) \big ) \nonumber \\
&&= \Big (  (\cosh (s)-1)+ 2 (s-\sinh (s)) (1-2 p)   \Big ) \big (\delta_i(p) \delta_i(p) \big ) \nonumber \\
&& = \big ( 2 s-2 \sinh (s)+\cosh (s)-1  -4 (s - \sinh(s)) p \big ) \big (\delta_i(p) \delta_i(p) \big ).  \nonumber
\end{eqnarray}
Using the identity $2 \sum \delta_i (p)^2 = (1-p) \triangle(p)- \triangle(p)p,$ we sum the above expressions 
and find that the formula \eqref{SCformula}, for the dilaton $h= s p$, reduces to 
\begin{eqnarray}
&&\frac{1}{2} (-2 \sinh (s)+\cosh (s)-1) \triangle(p) 
+ (-s+\sinh (s)-\frac{\cosh (s)}{2}+\frac{1}{2}) p \triangle(p) \nonumber \\
&& + \sinh ^2\left(\frac{s}{2}\right) \triangle(p) p 
+ 2 (s-\sinh (s)) p \triangle(p) p, \nonumber
\end{eqnarray}
up to multiplication from left by $e^{-h} = e^{-s p}$. 

\end{proof}

In contrast to the two dimensional case (cf. \cite{conmos2}), the functions of the 
variable $s \in \mathbb{R}$ that appear in the statement of Proposition \ref{SCprojection} are not bounded 
as they tend to $\pm \infty$ as $|s| \to  \infty$. The graphs of these functions are given below 
and some relations between these functions are investigated. First we graph $f_1$. 
\vskip 0.5 cm
\begin{center}
\includegraphics[scale=0.5]{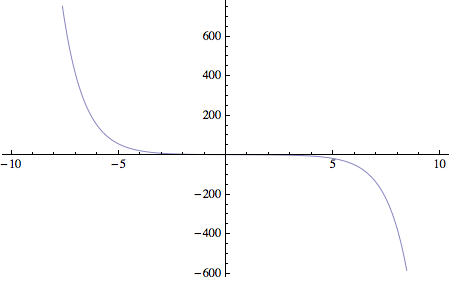}\\{Graph of the function $f_1(s)=\frac{1}{4} (-2 \sinh (s)+\cosh (s)-1)$.}
\label{GofT}
\end{center}

\vskip 0.5cm

Among these functions, $f_2$ is the only one that is bounded below, whereas the 
other functions are neither bounded above nor bounded below. In fact $f_2$ is a non-negative even function.
\begin{center}
\includegraphics[scale=0.5]{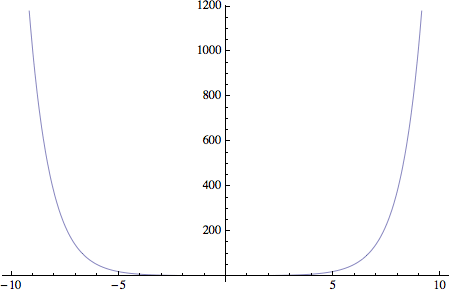}\\{Graph of the function $f_2(s)=\frac{1}{2} \sinh ^2\left(\frac{s}{2}\right)$.}
\label{GofT}
\end{center}
\vskip 0.5cm

The function $f_3$, similar to $f_1$,  does not satisfy any symmetry properties.  
\begin{center}
\includegraphics[scale=0.5]{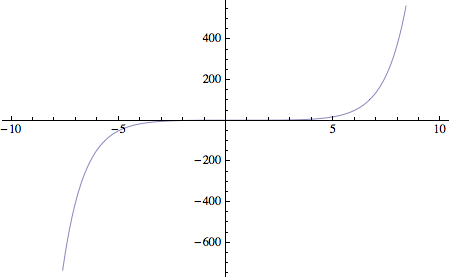}\\{Graph of the function $f_3(s)=\frac{1}{2} (\sinh (s)-s)+\frac{1}{4}
   (1-\cosh (s))$.}
\label{GofT}
\end{center}

\vskip 0.5cm
 
 The last function $f_4$ is obviously an odd function whose graph is given here: 
\begin{center}
\includegraphics[scale=0.5]{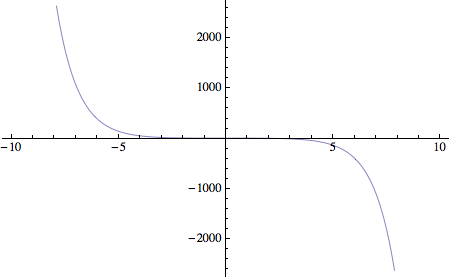}\\{Graph of the function $f_4(s)=2 (s-\sinh (s))$.}
\label{GofT}
\end{center}
\vskip 0.5cm
It is interesting to observe that these functions, which describe the scalar curvature for the dilaton $h= s p$, where $p$ is  
an arbitrary projection, satisfy the following relations: 

\[
f_1(s)+f_1(-s)=f_2(s)+f_2(-s)=- \big ( f_3(s)+f_3(-s)\big ) =-\sinh ^2\left(\frac{s}{2}\right), 
\]
\[
f_3(s)-f_3(-s)=-\frac{1}{2}  f_4(s) = \sinh (s)-s.
\]

\section{Gradient of the Einstein-Hilbert Action} \label{GRsection}

Denoting the Einstein-Hilbert action associated with the dilaton 
$h=h^* \in C^\infty(\mathbb{T}_\Theta^4)$ by $\Omega(h)$, we compute its 
gradient, namely an explicit formula for an element ${\rm Grad}_h \Omega$ that 
represents the derivative at $\varepsilon = 0$ of $\Omega(h+ \varepsilon a)$, 
where $h, a$ are selfadjoint smooth elements. 
The final formula for the gradient is expressed in terms of finite differences 
of the function $T$, obtained in the proof of Theorem 5.3 of \cite{fatkha2}. 
We recall from the proof of this theorem that 
\begin{eqnarray} \label{combined}
\Omega(h)=\varphi_0(R)=   
\sum_{i=1}^4  \varphi_0 \big ( e^{-h} T(\nabla)(\delta_i(h) ) 
\delta_i(h) \big ), \nonumber
\end{eqnarray}
where
\[
T(s)=\frac{-2 s+e^s-e^{-s} (2 s+3)+2}{4 s^2}.  
\]
The fact that this function is non-negative played a crucial role in identifying the extrema of the 
Einstein-Hilbert action in \cite{fatkha2}.

\begin{theorem} 
For any selfadjoint $h \in  C^\infty(\mathbb{T}_\Theta^4)$, we have 
\[
{\rm Grad}_h \Omega = \sum_{i=1}^4 \Big ( e^{-h} \omega_1(\nabla) (\delta_i^2(h)) + 
e^{-h} \omega_2(\nabla, \nabla) (\delta_i(h)^2) \Big ),   
\]
where 
\[
\omega_1(s)=\frac{-2 \sinh (s)+\sinh (2 s)-\cosh (2 s)+1}{4 s^2},
\]
\[
\omega_2(s, t)= 
\]
\begin{center}
\begin{math}
\frac{1}{4 s^2 t^2 (s+t)^2}
(-s (-(s^2+s t+3 t^2) \sinh
   (s)+(s^2+5 s t+t^2) \cosh (s)-t^2)
   (\sinh (3 (s+t))+\cosh (3 (s+t)))+s ((s^2 (8
   t+5)+s t (12 t+17)+t^2 (4 t+5)) (\sinh (s)+\cosh
   (s))-4 s t^2-4 t^3-5 t^2) (\sinh (s+t)+\cosh
   (s+t))+(-(5 s^3+s^2 t (4 t+15)+2 s t^2 (2
   t+5)+5 t^3) (\sinh (s)+\cosh (s))+t^2 (-(s+t))-t^2
   (4 s (s+t-2)-5 t) (\sinh (2 s)+\cosh (2 s))+t (s+t) (2
   s+t) (\sinh (3 s)+\cosh (3 s))) (\sinh (s+2
   t)+\cosh (s+2 t))+s (s \sinh (s)+s \cosh (s)+t) ((s+t)
   (\sinh (s)+\cosh (s))-t)) (\cosh (3 s+2 t)-\sinh (3
   s+2 t)).
   \end{math}
\end{center}

\begin{proof}

We have 

\begin{eqnarray}
&&\Omega(h+ \varepsilon a ) \nonumber \\ 
&&= \sum_{i=1}^4 \varphi_0 \big (  e^{-h - \varepsilon a }  \, T(\nabla_{h+ \varepsilon a}) 
(\delta_i(h+ \varepsilon a) ) \delta_i(h+ \varepsilon a) \big ) \nonumber \\ 
&&= 
\sum_{i=1}^4 \Big (  \varphi_0 \big (  e^{-h - \varepsilon a }  \, T(\nabla_{h+ \varepsilon a}) (\delta_i(h) ) \delta_i(h) \big ) + \varepsilon  \varphi_0 \big (  e^{-h - \varepsilon a }  \, T(\nabla_{h+ \varepsilon a}) (\delta_i(a) ) \delta_i(h) \big ) \nonumber \\
&&\quad  + \varepsilon  \varphi_0 \big (  e^{-h - \varepsilon a }  \, T(\nabla_{h+ \varepsilon a}) (\delta_i(h) ) \delta_i(a) \big )  + \varepsilon^2  \varphi_0 \big (  e^{-h - \varepsilon a }  \, T(\nabla_{h+ \varepsilon a}) (\delta_i(a) ) \delta_i(a) \big )  \Big ).  \nonumber 
\end{eqnarray}

Therefore 

\begin{eqnarray}
&&\frac{d}{d \varepsilon} \arrowvert_{\varepsilon=0} \Omega(h+ \varepsilon a ) \nonumber \\
&&=   \varphi_0 \big ( (\frac{d}{d \varepsilon} \arrowvert_{\varepsilon=0} e^{-h-\varepsilon a}) \, T(\nabla) (\delta(h)) \delta(h)  \big ) + 
\varphi_0 \big (  e^{-h} \, \frac{d}{d \varepsilon} \arrowvert_{\varepsilon=0} T(\nabla_{h+\varepsilon a} ) (\delta(h)) \delta(h)  \big ) \nonumber \\ 
&& \qquad + \varphi_0 \big (  e^{-h}  T(\nabla) (\delta(h)) \delta(a)  \big )+ \varphi_0 \big (  e^{-h}  T(\nabla) (\delta(a)) \delta(h)  \big ). \nonumber
\end{eqnarray}
Using the following lemmas we obtain the  explicit formula,  
\[
{\rm Grad}_h \Omega = \sum_{i=1}^4 \Big (e^{-h} \omega_1(\nabla) (\delta_i^2(h)) + 
e^{-h} \omega_2(\nabla, \nabla) (\delta_i(h) \delta_i(h)) \Big ), 
\]
where
\[ 
\omega_1(s) = - T(s) - T(-s) e^{-s}, 
\]
\begin{eqnarray}
\omega_2(s, t) &=& E(s, t) + L(s, t) - P(s, t) - Q(s, t) \nonumber \\
&=& \frac{e^{-s-t}-1}{s+t} T(s) +  e^{-s -t} \Big ( \frac{T(-t)-T(s)}{s+t} + \frac{T(t)-T(-s)}{s+t} e^t \Big ) \nonumber \\
&& - \Big (T(s+t) \frac{e^{-s}-1}{s}+\frac{T(t)-T(s+t)}{s} e^{-s}+\frac{T(s+t)-T(s)}{t} \Big ) \nonumber \\ 
&& - \Big ( T(-s-t) e^{-s-t} \frac{e^{-s}-1}{s} + \frac{T(t)-T(s+t)}{s} e^{-s} + \frac{T(s+t)-T(s)}{t} \Big ). \nonumber
\end{eqnarray}
Then one can find the above explicit functions in the statement of the theorem by direct computer assisted computations. 
\end{proof}

\end{theorem}  

For simplicity in the notation, in the following lemmas, $\delta$ can be taken to be any of the canonical derivations $\delta_i$ 
introduced in \S \ref{preliminaries}.  The proofs follow closely the 
techniques given in \cite{conmos2} for the computation of the gradient of linear functionals similar to $\Omega$ (see also \cite{les}). 

\begin{lemma}
We have
\[
\varphi_0 \big (    (\frac{d}{d \varepsilon} \arrowvert_{\varepsilon=0} e^{-h-\varepsilon a})  G(\nabla)(x) x  \big )
= 
\varphi_0 \big (  a e^{-h} E(\nabla, \nabla) (x x) \big ), 
\]
where 
\[
E(s, t)= \frac{e^{-s-t}-1}{s+t} G(s). 
\]
\begin{proof}
Using 
\[
\frac{d}{d \varepsilon}{\big |}_{\varepsilon=0} e^{-h- \varepsilon a} = \frac{1- e^\nabla}{\nabla}(a) e^{-h}, 
\]
we have 
\begin{eqnarray}
\varphi_0 \big (    (\frac{d}{d \varepsilon} \arrowvert_{\varepsilon=0} e^{-h-\varepsilon a})  G(\nabla)(x) x  \big )
&=& \varphi_0 \big ( \frac{1- e^\nabla}{\nabla}(a) e^{-h} G(\nabla)(x) x  \big )  \nonumber \\ 
&=& \varphi_0 \big ( a e^{-h} \,\frac{e^{-\nabla}-1}{\nabla} (G(\nabla)(x)x)\big ) \nonumber \\
&=& \varphi_0 \big (  a e^{-h} E(\nabla, \nabla) (x x) \big ). \nonumber
\end{eqnarray}
\end{proof}
\end{lemma}

\begin{lemma}
For any $x \in C^\infty(\mathbb{T}_\Theta^4)$, we have
\begin{eqnarray}
\varphi_0 \Big ( e^{-h} \frac{d}{d \varepsilon} \arrowvert_{\varepsilon=0} G(\nabla_{h+\varepsilon a})(x) x \Big ) = 
\varphi_0(a e^{-h}  \, L(\nabla, \nabla) (x  x )) \nonumber,
\end{eqnarray}
where 
\[
L(s, t) = e^{-s -t} \Big ( \frac{G(-t)-G(s)}{s+t} + \frac{G(t)-G(-s)}{s+t} e^t \Big ).
\]
\begin{proof}
Writing $G(v)=\int e^{-i t v} g(t) \, dt$ and using the following identity \cite{conmos2}  
\[ 
\frac{d}{d \varepsilon} \arrowvert_{\varepsilon=0} G(\nabla_{h+\varepsilon a}) = 
 \int_0^1 \int it \,\sigma_{ut} \,{\rm ad}_a \,\sigma_{(1-u)t} \,g(t) \,dt\, du, 
\]
we find that 
\[
\varphi_0 \Big ( e^{-h} \frac{d}{d \varepsilon} \arrowvert_{\varepsilon=0} G(\nabla_{h+\varepsilon a})(x) x \Big ) = 
\varphi_0(e^{-h} a \, L_0(\nabla, \nabla) (x  x )), 
\]
where 
\[
L_0(s, t) = \frac{G(-t)-G(s)}{s+t} + \frac{G(t)-G(-s)}{s+t} e^t.
\]
Therefore  
\[
\varphi_0 \Big ( e^{-h} \frac{d}{d \varepsilon} \arrowvert_{\varepsilon=0} G(\nabla_{h+\varepsilon a})(x) x \Big ) = 
\varphi_0( a e^{-h} \, L(\nabla, \nabla) (x  x )), 
\]
where
\[
L(s, t) = e^{-s-t} L_0(s, t). 
\]
\end{proof}

\end{lemma}

\begin{lemma}
 For any $x \in C^\infty(\mathbb{T}_\Theta^4)$ one has 
\[
\delta \big ( G(\nabla) (x) \big ) = G(\nabla)(\delta (x)) + M_1(\nabla, \nabla) (\delta(h) x) + M_2(\nabla, \nabla) (x \delta(h)), 
\] 
where 
\[
M_1(s, t)= \frac{G(t)-G(s+t)}{s}, \qquad M_2(s, t)=\frac{G(s+t)-G(s)}{t}. 
\]
\begin{proof}
It can be seen by writing $G(v)=\int e^{-i t v} g(t) \, dt$ and using the identity \cite{conmos2}
\[
\delta_i \sigma_t = \sigma_t \delta_i + it \int_0^1 \sigma_{ut} \, {\rm ad}_{\delta_i(h)} \sigma_{(1-u)t} \, du.
\]
\end{proof}
\end{lemma}

\begin{lemma}
We have 
\begin{eqnarray}
\varphi_0 \big ( e^{-h} G(\nabla)(\delta(h)) \delta(a) \big ) = - \varphi_0 \big ( a e^{-h} G(\nabla)(\delta^2(h)) \big ) 
- \varphi_0 \big ( a e^{-h} P(\nabla, \nabla) (\delta(h) \delta(h))  \big ), \nonumber
\end{eqnarray}
where 
\[
P(s, t)= G(s+t) \frac{e^{-s}-1}{s}+M_1(s, t)e^{-s}+M_2(s,t). 
\]
\begin{proof}
We start by writing 
\begin{eqnarray}
&& \varphi_0 \big ( e^{-h} G(\nabla)(\delta(h)) \delta(a) \big ) = - \varphi_0  \big (a\, \delta ( G(\nabla)( e^{-h} \delta(h)) )   \big ) \nonumber \\ 
 &&= - \varphi_0  \big (a\, G(\nabla)  ( \delta ( e^{-h} \delta(h)) )   \big ) - \varphi_0 \big( a M_1(\nabla, \nabla) (\delta(h) e^{-h} \delta(h))   \big ) \nonumber \\
&& \quad - \varphi_0 \big( a M_2(\nabla, \nabla) ( e^{-h} \delta(h) \delta(h))   \big ) \nonumber \\ 
&&= - \varphi_0  \big (a e^{-h}\, G(\nabla)  ( e^{h} \delta ( e^{-h}) \delta(h) )   \big ) - \varphi_0  \big (a\, G(\nabla)  (  e^{-h} \delta^2(h)) )   \big ) \nonumber \\
&&\quad - \varphi_0 \big( a e^{-h} M_1(\nabla, \nabla) (e^{-\nabla} (\delta(h))  \delta(h))   \big )  - \varphi_0 \big( a e^{-h}  M_2(\nabla, \nabla) ( \delta(h) \delta(h))   \big ). \nonumber 
\end{eqnarray}
Then, using the fact that $e^h \delta(e^{-h})= \frac{e^{-\nabla}-1}{\nabla}(\delta(h))$, one can find the above expression for $\varphi_0 \big ( e^{-h} G(\nabla)(\delta(h)) \delta(a) \big )$ in the statement of the lemma.
 
\end{proof}

\end{lemma}

\begin{lemma}
We have
\begin{eqnarray}
\varphi_0 \big ( e^{-h} G(\nabla)(\delta(a)) \delta(h) \big ) = 
- \varphi_0 \big ( a e^{-h} \bar G (\nabla)(\delta^2(h)) \big ) 
- \varphi_0 \big ( a e^{-h} Q(\nabla, \nabla) (\delta(h) \delta(h))  \big ), \nonumber
\end{eqnarray}
where 
\begin{eqnarray}
\bar G(s) &=& G(-s)e^{-s}, \nonumber \\ 
Q(s, t) &=& \bar G(s+t) \frac{e^{-s}-1}{s}+M_1(s, t)e^{-s}+M_2(s,t). \nonumber
\end{eqnarray}
\begin{proof}
It follows from the previous lemma after writing 
\begin{eqnarray}
\varphi_0 \big ( e^{-h} G(\nabla)(\delta(a)) \delta(h) \big ) &=& \varphi_0 \big ( e^{-h} G(-\nabla) e^{-\nabla}(\delta(h)) \delta(a) \big ).
 \nonumber 
\end{eqnarray}

\end{proof}

\end{lemma}

The Taylor series at $s=0$ of the one variable function $\omega_1$ appearing in the formula for the 
gradient  of $\Omega$ is given by  
\[ \omega_1(s)=-\frac{1}{2}+\frac{s}{4}-\frac{s^2}{6}+\frac{s^3}{16}-\frac{
   s^4}{45}+\frac{s^5}{160}+O\left(s^6\right), 
\]
and the following is its graph.  
\vskip 0.5 cm
\begin{center}
\includegraphics[scale=0.5]{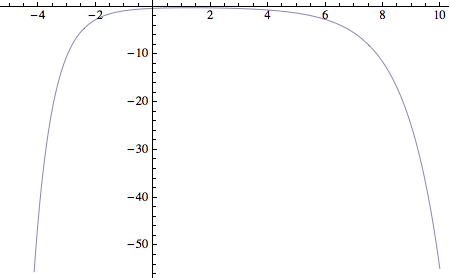}\\{Graph of the function $\omega_1$.}
\label{GofK}
\end{center}
\vskip 0.5 cm

The function of two variables $\omega_2$ appearing in the formula for the 
above gradient has the following Taylor expansion at the origin 

\begin{eqnarray} 
\omega_2(s, t) &=&
\left(\frac{1}{4}+\frac{t}{24}+\frac{13 t^2}{240}-\frac{7
   t^3}{360}+O\left(t^4\right)\right) +s
   \left(-\frac{3}{8}+\frac{5 t}{48}-\frac{17
   t^2}{120}+\frac{t^3}{14}+O\left(t^4\right)\right) \nonumber \\
   &&+s^2
   \left(\frac{47}{240}-\frac{t}{6}+\frac{77
   t^2}{480}-\frac{1159
   t^3}{13440}+O\left(t^4\right)\right)\nonumber \\ 
   &&+s^3
   \left(-\frac{83}{720}+\frac{169 t}{1260}-\frac{697
   t^2}{5760}+\frac{151
   t^3}{2304}+O\left(t^4\right)\right)+O\left(s^4\right),  \nonumber 
\end{eqnarray}   
 and here is its graph:

\vskip 0.5 cm
\begin{center}
\includegraphics[scale=0.5]{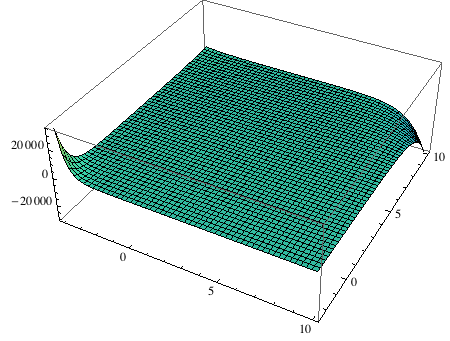}\\{Graph of the function $\omega_2$.}
\label{GofH}
\end{center}
\vskip 0.5 cm

We now look at the behavior of the function $\omega_2$ on the diagonals. We have 
\[
\omega_2(s, s) =
\]
{\small \[
-\frac{e^{-3 s/2} \sinh \left(\frac{s}{2}\right) (8 s+(8 s-5)
   \sinh (s)-3 \sinh (2 s)+\sinh (3 s)-8 \cosh (s)-3 \cosh (2
   s)+11)}{4 s^3}, 
\]
}
with the Taylor expansion 
\[
\omega_2(s, s) =
\frac{1}{4}-\frac{s}{3}+\frac{17 s^2}{48}-\frac{319
   s^3}{720}+\frac{623 s^4}{1440}-\frac{155
   s^5}{448}+O\left(s^6\right), 
\]
and the following graph. 
\vskip 0.5 cm
\begin{center}
\includegraphics[scale=0.5]{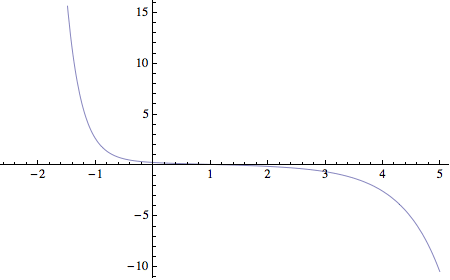}\\{Graph of the function $s \mapsto \omega_2(s, s)$.}
\label{GofH}
\end{center}
\vskip 0.5 cm
On the other diagonal we have 
\[
\omega_2(s, -s)=\frac{4 s+e^{-2 s}-2 e^s+1}{4 s^2}=\frac{1}{4}-\frac{5 s}{12}+\frac{7 s^2}{48}-\frac{17
   s^3}{240}+\frac{31 s^4}{1440}-\frac{13
   s^5}{2016}+O\left(s^6\right), 
\]
whose graph is the following. 
\vskip 0.5 cm
\begin{center}
\includegraphics[scale=0.5]{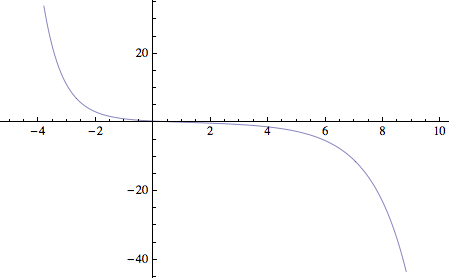}\\{Graph of the function $s \mapsto \omega_2(s, -s)$.}
\label{GofH}
\end{center}
\vskip 0.5 cm

\section{Discussion}

Considering the variety of geometric spaces that fit into the paradigm of noncommutative 
geometry \cite{con1.5, con3}, it is of great importance to 
develop different methods for computing their local geometric invariants such as 
scalar curvature. Different methods of computation can also help to achieve conceptual understandings  
of such invariants for specific examples.

The appearance of functions of a modular automorphism in the final formulas 
for the scalar curvature of noncommutative tori  is a purely 
noncommutative feature \cite{conmos2, fatkha4, fatkha2, dabsit2}, which accompanies the 
following striking facts. First, the final one and two 
variable functions in the curvature formulas are significantly simple, which indicates an enormous amount 
of cancellations in the algebraically lengthy formulas that involve hundreds of terms with numerous 
functions from the rearrangement lemma  involved \cite{contre, conmos2, bhumar, fatkha4}. Second, 
the function of two variables for the curvature of $\mathbb{T}_\theta^2$ can be recovered from the one 
variable function \cite{conmos2} by finite differences.

Using the noncommutative residue that involves integration on the 3-sphere \cite{fatwon, fatkha2}, 
we computed the scalar curvature of $\mathbb{T}_\Theta^4$ in this paper without using the rearrangement 
lemma. This method avoids the complexity stemming from the functions coming from this lemma and also 
explains the simplicity of the final functions in the curvature formula. It  should be emphasized that there is 
another rather technical simplifying factor in this method, compared to the use of parametric pseudodifferential calculus, 
which is due to the fact that the first term $b_0$ of the parametrix of the Laplacian reduces on the sphere to a power of the Weyl factor. 
It can be seen in the derivation of \eqref{SC4formula1} that this softens out some further complexities and leads to 
the final expression, whose summand consists of only a few terms. By working out a simple relation between 
the functions that relate the derivatives of the Weyl factor to the derivatives of the dilaton, we have then written 
the two variable function in the formula for the curvature of $\mathbb{T}_\Theta^4$ as the sum of a finite difference and a finite product of the one variable function.

Similar to the concrete illustration in \cite{conmos2} of the scalar curvature of $\mathbb{T}_\theta^2$ for dilatons associated 
with projections, which exist in abundance for noncommutative tori \cite{rie1}, we have worked out a concrete 
formula in the case of $\mathbb{T}_\Theta^4$. For a dilaton of the form $h = s p$, where $s \in \mathbb{R}$ and 
$p\in C^\infty(\mathbb{T}_\Theta^4)$ is an arbitrary projection, the final concrete formula for the curvature involves 
unbounded functions of the parameter $s$, which is in contrast to the striking fact about the boundedness of the 
functions obtained in the two dimensional case \cite{conmos2}. The question that arises is whether there is a conceptual 
meaning behind this contrast. Also, the explicit computation of the gradient of the analog of the Einstein-Hilbert action 
for $\mathbb{T}_\Theta^4$ prepares the ground for further studies of the natural associated geometric flows in this context, 
cf. \cite{bhumar, conmos2}.

\section*{Acknowledgments}

The author thanks the Max-Planck Institute for Mathematics, the Stefan Banach International Mathematical Center at 
the Institute for Mathematics Polish Academy of Sciences, and the Hausdorff Research Institute for Mathematics 
for their support and hospitality.

\end{document}